\documentclass[12pt,a4paper,reqno]{amsart}
\allowdisplaybreaks
\usepackage{amsmath}
\usepackage{amsfonts}
\usepackage{amssymb,amsthm,amsfonts,amsthm,latexsym,enumerate,url,cases}
\numberwithin{equation}{section}
\usepackage{mathrsfs}
\usepackage{hyperref}
\hypersetup{colorlinks=true,citecolor=blue,linkcolor=blue,urlcolor=blue}
\def\pmod #1{\ ({\rm{mod}}\ #1)}

\theoremstyle{plain}
\newtheorem{theorem}{Theorem}

\newtheorem{problem}{Problem}

\theoremstyle{definition}

\usepackage{etoolbox}
\makeatletter
\patchcmd{\@settitle}{\uppercasenonmath\@title}{}{}{}
\patchcmd{\@setauthors}{\MakeUppercase}{}{}{}
\patchcmd{\section}{\scshape}{}{}{}
\makeatother

\begin{document}

\title
[{Note on a problem of S\'ark\"ozy on multiplicative representation functions}]
{Note on a problem of S\'ark\"ozy on multiplicative representation functions}

\author
[Y. Ding] {Yuchen Ding}

\address{(Yuchen Ding) School of Mathematics,  Yangzhou University, Yangzhou 225002, People's Republic of China}
\email{ycding@yzu.edu.cn}

\keywords{Dirichlet's theorem in arithmetic progressions, divisors.}
\subjclass[2010]{11A05}

\begin{abstract}
Motivated by a 2001 problem of S\'ark\"ozy, we classify all situations of the integers $b,c,e$ and $f$ satisfying 
\begin{align*}
\limsup_{n\rightarrow\infty}|d(\mathcal{A},bn+c)-d(\mathcal{A},en+f)|=\infty
\end{align*}
for any infinite $\mathcal{A}\subset \mathbb{N}$, where 
$
d(\mathcal{A},m)=\#\{a\in \mathcal{A}:a|m\}.
$
\end{abstract}
\maketitle

Let $\mathbb{N}$ be the set of natural numbers. In his unsolved problems list, S\'ark\"ozy  \cite[Problem 25]{Sarkozy} asked the following interesting question.

\begin{problem}[S\'ark\"ozy]\label{pro1}
For $\mathcal{A}\subset \mathbb{N}$ and $n\in \mathbb{N}$ let 
$
d(\mathcal{A},n)=\#\{a\in \mathcal{A}:a|n\}.
$
Then is it true that for any infinite set $\mathcal{A}\subset \mathbb{N}$, $|d(\mathcal{A},n+1)-d(\mathcal{A},n)|$ cannot be bounded?
\end{problem}

Let $a_1<a_2<\cdots<a_k$ be the first $k$ elements of $\mathcal{A}$. Then by Dirichlet's theorem in arithmetic progressions \cite[Theorem 7.9]{Apostol} there is a prime $\ell_k$ satisfying 
$$
\ell_k\equiv -1\pmod{a_1a_2\cdots a_k}.
$$
 Hence, we have
$$
\big|d(\mathcal{A}, \ell_k+1)-d(\mathcal{A}, \ell_k)\big|\ge k-2,
$$
from which the affirmative answer to Problem 25 of S\'ark\"ozy clearly follows. In another note \cite{Ding}, I gave a short answer to Problem 26 of S\'ark\"ozy in the same problem list.

Let
$$
f_{\mathcal{A}}(n)=\min_{1\le i\neq j\le 3}\big\{|d(\mathcal{A},n+i)-d(\mathcal{A},n+j)|\big\}.
$$
I asked Prof. Yong-Gao Chen whether $\limsup_{n\rightarrow\infty}f_{\mathcal{A}}(n)=\infty$ for any infinite $\mathcal{A}\subset \mathbb{N}$? Immediately, he gave a negative answer  by taking $A=\{4m:m\in \mathbb{N}\}$ since at least two of $n+1, n+2$ and $n+3$ can not be divisible by $4$. Chen then suggested the following generalized problem. 

\begin{problem}[Chen]\label{pro2}
Determine the sufficient and necessary conditions of integers $b,c,e$ and $f$ satisfying 
\begin{align*}
\limsup_{n\rightarrow\infty}|d(\mathcal{A},bn+c)-d(\mathcal{A},en+f)|=\infty
\end{align*}
for any infinite $\mathcal{A}\subset \mathbb{N}$.
\end{problem}

The following theorem which is a generalization of S\'ark\"ozy's problem provides a complete solution to Problem \ref{pro2}.

\begin{theorem}\label{thm2}
Let $b,c,e$ and $f$ be integers. Then we have
\begin{align}\label{eq4}
\limsup_{n\rightarrow\infty}|d(\mathcal{A},bn+c)-d(\mathcal{A},en+f)|=\infty
\end{align}
for any infinite $\mathcal{A}\subset \mathbb{N}$ 
if and only if one of the following cases holds:

I. $b=c=e=0$ and $f\neq 0$; or $b=c=f=0$ and $e\neq 0$;

II. $b=0$, $e\neq 0$, and $c=0$; or $b\neq 0$, $e= 0$, and $f=0$;

III. $b=0$, $e\neq 0$, $c\neq 0$, and $e|f$; or $b\neq 0$, $e= 0$, $f\neq 0$, and $b|c$;

IV. $be \neq 0$, $bf\neq ec$, and $b|c$; or $be \neq 0$, $bf\neq ec$, and $e|f$.
\end{theorem}

\begin{proof}[Proof of Theorem \ref{thm2}]
{\bf Necessities.}
To prove the necessities, we separate them into a few cases.

{\bf Case 1.} $b=e=0$.

It is clear that the situation $b=c=e=f=0$ does not satisfy our requirement. Moreover, at least one of $c$ and $f$ equals zero. Otherwise, we have
\begin{align*}
\limsup_{n\rightarrow\infty}|d(\mathcal{A},c)-d(\mathcal{A},f)|\le |c|+|f|<\infty
\end{align*}
for any infinite $\mathcal{A}\subset \mathbb{N}$ which is a contradiction. ({\bf Case 1.} corresponds with I.)

{\bf Case 2.} $b=0$ and $e\neq 0$; or $b\neq 0$ and $e= 0$.

We only consider the case $b=0$ and $e\neq 0$ since the other situation will be treated similarly. In this case, $c=0$ is admissible since $d(\mathcal{A},en+f)<\infty$ for all sufficiently large $n$ as $en+f\neq 0$ for $n>-f/e$. We are left over to consider the situation $b=0,e\neq 0$, and $c\neq 0$. For any infinite $\mathcal{A}\subset \mathbb{N}$ we clearly have
$
d(\mathcal{A},c)\le |c|.
$
We are going to show $e|f$. Assume the contrary, i.e., $e\nmid f$. Take $\mathcal{A}_e=\{|e|m: m\in \mathbb{N}\}\subset \mathbb{N}$ and then
$$
\limsup_{n\rightarrow\infty}d(\mathcal{A}_e,en+f)=0.
$$ 
It follows that
\begin{align*}
\limsup_{n\rightarrow\infty}|d(\mathcal{A}_e,c)-d(\mathcal{A}_e,en+f)|\le |c|<\infty,
\end{align*}
which is a contradiction. ({\bf Case 2.} corresponds with II and III.)

{\bf Case 3.} $be \neq 0$.

In this case, we firstly show that $c^2+f^2\neq 0$, i.e., at least one of $c$ and $f$ is nonzero. Otherwise, we choose a prime $p_0$ so that $p_0\nmid b$ and $p_0\nmid e$ and take 
\begin{align}\label{newadded1}
\mathcal{A}_0=\{p_0^m: m\in \mathbb{N}\}\subset \mathbb{N}.
\end{align} 
Then we have
\begin{align*}
\limsup_{n\rightarrow\infty}|d(\mathcal{A}_0,bn)-d(\mathcal{A}_0,en)|&=\limsup_{n\rightarrow\infty}|d(\mathcal{A}_0,n)-d(\mathcal{A}_0,n)|=0,
\end{align*} 
which is a contradiction. 

Next, we prove $b|c$ or $e|f$. Assume the contrary. Then we may assume $b=xb_1, c=xc_1, e=ye_1$, and $f=yf_1$ with $b_1,e_1\neq 1$, $\gcd(b_1,c_1)=1$, and $\gcd(e_1,f_1)=1$. We take $\mathcal{A}_{b_1,e_1}=\{(b_1e_1)^m: m\in \mathbb{N}\}\subset \mathbb{N}$. We claim that $d(\mathcal{A}_{b_1,e_1},bn+c)$ and $d(\mathcal{A}_{b_1,e_1},en+f)$ are bounded for all $n$, from which we have
$$
\limsup_{n\rightarrow\infty}|d(\mathcal{A}_{b_1,e_1},bn+c)-d(\mathcal{A}_{b_1,e_1},en+f)|<\infty,
$$
leading to a contradiction with (\ref{eq4}). Actually, for all sufficiently large $m$ if $(b_1e_1)^m|bn+c$, then $b_1|b_1n+c_1$. This is a contradiction with $\gcd(b_1,c_1)=1$ and $b_1>1$. Thus, the function $d(\mathcal{A}_{b_1,e_1},bn+c)$ is bounded for all $n$. The same argument yields that $d(\mathcal{A}_{b_1,e_1},en+f)$ is bounded for all $n$.

We are left over to show $bf\neq ec$. Assume the contrary, i.e., 
$$
\frac{c}{b}=\frac{f}{e}=g,\quad \text{say.}
$$
Since $c^2+f^2\neq 0$, we have $g\neq 0$. Moreover, $g\in \mathbb{Z}$ since $b|c$ or $e|f$. Then for $\mathcal{A}_0$ as in (\ref{newadded1}) we also have
\begin{align*}
\limsup_{n\rightarrow\infty}|d(\mathcal{A}_0,bn+c)-d(\mathcal{A}_0,en+f)|&=\limsup_{n\rightarrow\infty}|d(\mathcal{A}_0,b(n+g))-d(\mathcal{A}_0,e(n+g))|\\
&=\limsup_{n\rightarrow\infty}|d(\mathcal{A}_0,n+g)-d(\mathcal{A}_0,n+g)|=0,
\end{align*} 
which is a contradiction. 
({\bf Case 3.} corresponds with IV.)

{\bf Sufficiencies.} We now turn to the proof of sufficiencies case by case. 

{\bf I. implies (\ref{eq4}).} Clearly.

{\bf II. implies (\ref{eq4}).} Without loss of generality, we only consider the case $b=0$, $e\neq 0$, and $c=0$ since the other one can be proved similarly. We have $d(\mathcal{A},en+f)<\infty$ for any $n\neq -f/e$ since $e\neq 0$. But $d(\mathcal{A},0)=\infty$ for any $n$, from which $(\ref{eq4})$ clearly follows.

{\bf III. implies (\ref{eq4}).} Without loss of generality, we only consider the case $b=0$, $e\neq 0$, $c\neq 0$, and $e|f$ since the other one can be proved similarly. We assume $f=eh$. For any positive integer $t$ let
$$
P_t=\prod_{1\le i\le t}a_i.
$$
For any $t$ there is a positive integer $n_t$ so that
$$
n_t\equiv -h\pmod {P_{t}},
$$
which means that 
$
P_t|e(n_t+h).
$
In other words, $d(\mathcal{A},n_t+h)\ge t$.
Hence, for sufficiently large $t$ we have 
\begin{align*}
\limsup_{n\rightarrow\infty}|d(\mathcal{A},c)-d(\mathcal{A}_g,en+f)|&\ge\limsup_{t\rightarrow\infty}|d(\mathcal{A},c)-d(\mathcal{A},e(n_t+h))|\\
&\ge \limsup_{t\rightarrow\infty}(t-|c|)=\infty.
\end{align*} 

{\bf IV. implies (\ref{eq4}).} Without loss of generality, we only consider the case $be \neq 0$, $bf\neq ec$, and $b|c$ since the other one can be proved similarly.

Suppose that $\mathcal{A}=\{a_1<a_2<\cdots<a_i<\cdots\}$ and $c=bh$. Since $bf\neq ec$, we have $f-eh\neq 0$. For any positive integer $t$ let
$$
P_t=\prod_{1\le i\le t}a_i \quad \text{and} \quad \ell=\sup_{t\ge 1}\big\{\gcd(eP_t,f-eh)\big\}.
$$
Then there is some fixed $t_0$ so that $\ell|eP_{t_0}$ and
\begin{align}\label{newadded2}
P_{t_0}>|h|. 
\end{align}
For any positive integer $m$, we  have 
$$
\gcd\left(\frac{eP_{t_0+m}}{\ell},\frac{f-eh}{\ell}\right)=1.
$$
Now, we will separate the arguments into two cases. 

Suppose first that $e>0$.  
In this case, by Dirichlet's theorem in arithmetic progressions there exists a prime 
\begin{align}\label{newadded3}
p_m>\max\Big\{\frac{|f-eh|}{\ell},\ell\Big\}
\end{align}
satisfying that
\begin{align}\label{eq5}
p_m\equiv \frac{f-eh}{\ell} \pmod{eP_{t_0+m}/\ell},
\end{align}
which clearly means that
\begin{align}\label{eq6}
\ell p_m-f+eh=eP_{t_0+m}q_m
\end{align}
for some integer $q_m$. Note that $p_m>\frac{|f-eh|}{\ell}$ from (\ref{newadded3}). We have $\ell p_m-f+eh>0$ and hence $P_{t_0+m}q_m>0$.
Rearranging the terms in (\ref{eq6}) leads to
\begin{align}\label{eq7}
e(P_{t_0+m}q_m-h)+f=\ell p_m.
\end{align}
Moreover, one notes easily that
\begin{align}\label{eq8}
b(P_{t_0+m}q_m-h)+c=bP_{t_0+m}q_m\equiv 0\pmod{P_{t_0+m}}.
\end{align}
From (\ref{eq7}) and (\ref{eq8}) we have
\begin{align}\label{eq9}
d(\mathcal{A},e(P_{t_0+m}q_m-h)+f)=d(\mathcal{A},\ell p_m)
\end{align}
and
\begin{align}\label{eq10}
d(\mathcal{A},b(P_{t_0+m}q_m-h)+c)\ge t_0+m.
\end{align}
Recall that $p_m>\ell$ for any $m$ from (\ref{newadded3}). Hence 
\begin{align}\label{eq11}
d(\mathcal{A},\ell p_m)\le 2\ell\le 2|f-eh|.
\end{align}
Furthermore, from (\ref{newadded2}) we have 
$
P_{t_0+m}q_m-h>P_{t_0}-h>0.
$
We conclude from (\ref{eq9}), (\ref{eq10}) and (\ref{eq11}) that for sufficiently large $m$,
\begin{align*}
\limsup_{n\rightarrow\infty}|&d(\mathcal{A},bn+c)-d(\mathcal{A},en+f)|\\
&\ge\limsup_{m\rightarrow\infty}|d(\mathcal{A},b(P_{t_0+m}q_m-h)+c)-d(\mathcal{A},e(P_{t_0+m}q_m-h)+f)|\\
&=\limsup_{m\rightarrow\infty}|d(\mathcal{A},b(P_{t_0+m}q_m-h)+c)-d(\mathcal{A},\ell p_m)|\\
&\ge \limsup_{m\rightarrow\infty}\big(t_0+m-2|f-eh|\big)=\infty.
\end{align*}

It remains to consider the case $e<0$. In this case, we simply replace (\ref{eq5}) by
$$
p_m\equiv -\frac{f-eh}{\ell} \pmod{eP_{t_0+m}/\ell}.
$$
Then (\ref{eq6}) will be changed to
$$
-\ell p_m-f+eh=eP_{t_0+m}q_m.
$$
We see easily that $q_m$ is also positive. Adjusting the arguments accordingly would yield our desired result. 
\end{proof}

Motivated by Chen's problem, I have the following extended one.
Let $F(n)$ and $G(n)$ be integer polynomials of degrees $r$ and $s$, respectively.

\begin{problem}[S\'ark\"ozy's Problem 25 revisited]\label{problem-3}
Let $r$ and $s$ be given integers. Classify all the cases of $F(n)$ and $G(n)$ such that
\begin{align*}
\limsup_{n\rightarrow\infty}\big|d\big(\mathcal{A},F(n)\big)-d\big(\mathcal{A},G(n)\big)\big|=\infty
\end{align*}
for any infinite $\mathcal{A}\subset \mathbb{N}$.
\end{problem}

\medskip

Theorem \ref{thm2} provides the answer to Problem \ref{problem-3} for $r, s\le 1$.

\section*{Acknowledgments}
The author thanks Yong-Gao Chen and Honghu Liu for their interest in this work. Thanks also go the anonymous referees for their helpful comments.

The author is supported by National Natural Science Foundation of China  (Grant No. 12201544) and China Postdoctoral Science Foundation (Grant No. 2022M710121).


\begin{thebibliography}{1}
\bibitem{Apostol} 
T. M. Apostol, {\it Introduction to analytic number theory,} Undergraduate Texts in Mathematics, Springer, 1976, xii+338 pages.

\bibitem{Ding} 
Y. Ding, {\it On Problem 26 of S\'ark\"ozy's list,} to appear in Period. Math. Hungar.

\bibitem{Sarkozy} 
A. S\'ark\"ozy, {\it Unsolved problems in number theory,} Period. Math. Hungar. {\bf 42} (2001), 17--35.

\end{thebibliography}
\end{document}